\documentclass[12pt]{amsart}
\usepackage{hyperref}
\usepackage[all]{xy}

\newtheorem{theorem}{Theorem}[section]
\newtheorem{definition}[theorem]{Definition}
\newtheorem{proposition}[theorem]{Proposition}

\begin{document}

\title{Truncation and duality results for Hopf image algebras}

\author{Teodor Banica}
\address{T.B.: Department of Mathematics, Cergy-Pontoise University, 95000 Cergy-Pontoise, France. {\tt teodor.banica@u-cergy.fr}}

\subjclass[2000]{46L65 (46L37)}
\keywords{Quantum permutation, Hadamard matrix}

\begin{abstract}
Associated to an Hadamard matrix $H\in M_N(\mathbb C)$ is the spectral measure $\mu\in\mathcal P[0,N]$ of the corresponding Hopf image algebra, $A=C(G)$ with $G\subset S_N^+$. We study here a certain family of discrete measures $\mu^r\in\mathcal P[0,N]$, coming from the idempotent state theory of $G$, which converge in Ces\`aro limit to $\mu$. Our main result is a duality formula of type $\int_0^N(x/N)^pd\mu^r(x)=\int_0^N(x/N)^rd\nu^p(x)$, where $\mu^r,\nu^r$ are the truncations of the spectral measures $\mu,\nu$ associated to $H,H^t$. We prove as well, using these truncations $\mu^r,\nu^r$, that for any deformed Fourier matrix $H=F_M\otimes_QF_N$ we have $\mu=\nu$.
\end{abstract}

\maketitle

\section*{Introduction}

A complex Hadamard matrix is a square matrix $H\in M_N(\mathbb C)$ whose entries are on the unit circle, $|H_{ij}|=1$, and whose rows are pairwise orthogonal. The basic example of such a matrix is the Fourier one, $F_N=(w^{ij})$ with $w=e^{2\pi i/N}$:
$$F_N=\begin{pmatrix}
1&1&1&\ldots&1\\
1&w&w^2&\ldots&w^{N-1}\\
\ldots&\ldots&\ldots&\ldots&\ldots\\
1&w^{N-1}&w^{2(N-1)}&\ldots&w^{(N-1)^2}
\end{pmatrix}$$

In general, the theory of complex Hadamard matrices can be regarded as a ``non-standard'' branch of discrete Fourier analysis. For a number of potential applications to quantum physics and quantum information theory questions, see \cite{ben}, \cite{tzy}, \cite{wer}.

Each Hadamard matrix $H\in M_N(\mathbb C)$ is known to produce a subfactor $M\subset R$ of the Murray-von Neumann hyperfinite factor $R$, having index $[R:M]=N$. The associated planar algebra $P=(P_k)$ has a direct description in terms of $H$, worked out in \cite{jsu}, and a key problem is that of computing the corresponding Poincar\'e series, given by:
$$f(z)=\sum_{k=0}^\infty\dim(P_k)z^k$$

An alternative approach to this question is via quantum groups \cite{wo1}, \cite{wo2}. The idea is that associated to $H\in M_N(\mathbb C)$ is a quantum subgroup $G\subset S_N^+$ of Wang's quantum permutation group \cite{wan}, constructed by using to the Hopf image method, developed in \cite{bbi}. More precisely, $G\subset S_N^+$ appears via a factorization diagram, as follows:
$$\xymatrix{C(S_N^+)\ar[rr]^\pi\ar[rd]&&M_N(\mathbb C)\\&C(G)\ar[ur]_\rho&}$$

Here the upper arrow is defined by $\pi:u_{ij}\to P_{ij}=Proj(H_i/H_j)$, where $u_{ij}$ are the standard generators of $C(S_N^+)$, and where $H_1,\ldots,H_N\in\mathbb T^N$ are the rows of $H$. The lower left arrow is by definition transpose to the embedding $G\subset S_N^+$, and the quantum group $G\subset S_N^+$ itself is by definition the minimal one producing such a factorization.

With this notion in hand, the problem is that of computing the spectral measure $\mu$ of the main character $\chi:G\to\mathbb C$. This is indeed the same problem as above, because by Woronowicz's Tannakian duality \cite{wo2}, $f$ is the Stieltjes transform of $\mu$:
$$f(z)=\int_G\frac{1}{1-z\chi}$$

Here, and in what follows, we use the integration theory developed in \cite{wo1}.

As a basic example, for a Fourier matrix $F_N$ the associated quantum group $G\subset S_N^+$ is the cyclic group $\mathbb Z_N$, and we therefore have $\mu=(1-\frac{1}{N})\delta_0+\frac{1}{N}\delta_N$ in this case. In general, however, the computation of $\mu$ is a quite difficult question. See \cite{bfs}.

In this paper we discuss a certain truncation procedure for the main spectral measure, coming from the idempotent state theory of the associated quantum group \cite{bfs}, \cite{fsk}. Consider the following functionals, where $*$ is the convolution, $\psi*\phi=(\psi\otimes\phi)\Delta$:
$$\int_G^r=(tr\circ\rho)^{*r}$$

The point with these functionals is that, as explained in \cite{bfs}, we have the following Ces\`aro limiting result, coming from the general results of Woronowicz in \cite{wo1}:
$$\int_G\varphi=\lim_{k\to\infty}\frac{1}{k}\sum_{r=1}^k\int_G^r\varphi$$

This formula can be of course used in order to estimate or exactly compute the various integrals over $G$, and doing so will be the main idea in the present paper. 

At the level of the main character, we have the following result:

\medskip

\noindent {\bf Theorem A.} {\em The law $\chi$ with respect to $\int_G^r$ equals the law of the Gram matrix
$$X_{i_1\ldots i_r,j_1\ldots j_r}=<\xi_{i_1\ldots i_r},\xi_{j_1\ldots j_r}>$$
of the norm one vectors $\xi_{i_1\ldots i_r}=\frac{1}{\sqrt{N}}\cdot\frac{H_{i_1}}{H_{i_2}}\otimes\ldots\otimes\frac{1}{\sqrt{N}}\cdot\frac{H_{i_r}}{H_{i_1}}$.}

\medskip

Here the law of $X$ is by definition its spectral measure, with respect to the trace.

Observe that with $r\to\infty$, via the above-mentioned Ces\`aro limiting procedure, we obtain from the laws in Theorem A the spectral measure $\mu$ that we are interested in.

Our second, and main theoretical result, is as follows:

\medskip

\noindent {\bf Theorem B.} {\em We have the moment/truncation duality formula
$$\int_{G_H}^r\left(\frac{\chi}{N}\right)^p=\int_{G_{H^t}}^p\left(\frac{\chi}{N}\right)^r$$
where $G_H,G_{H^t}$ are the quantum groups associated to $H,H^t$.}

\medskip

This formula, which is quite non-trivial, is probably quite interesting, in connection with the duality between the quantum groups $G_H,G_{\overline{H}},G_{H^t},G_{H^*}$ studied in \cite{ban}.

As an illustration for the above methods, we will work out the case of the deformed Fourier matrices, $H=F_N\otimes_QF_M$, with the following result:

\medskip

\noindent {\bf Theorem C.} {\em For $H=F_N\otimes_QF_M$ we have the self-duality formula
$$\int_{G_H}\varphi(\chi)=\int_{G_{H^t}}\varphi(\chi)$$
valid for any parameter matrix $Q\in M_{M\times N}(\mathbb T)$.}

\medskip

The paper is organized as follows: 1-2 are preliminary sections, and in 3-4-5 we present the truncation procedure, and we prove Theorems A-B-C above.

\medskip

\noindent {\bf Acknowledgements.} I would like to thank Julien Bichon, Pierre Fima, Uwe Franz, Adam Skalski and Roland Vergnioux for several interesting discussions.

\section{Hadamard matrices}

A complex Hadamard matrix is a matrix $H\in M_N(\mathbb C)$ whose entries are on the unit circle, and whose rows are pairwise orthogonal. The basic example is the Fourier matrix, $F_N=(w^{ij})$ with $w=e^{2\pi i/N}$.  More generally, we have as example the Fourier matrix $F_G=F_{N_1}\otimes\ldots\otimes F_{N_k}$ of any finite abelian group $G=\mathbb Z_{N_1}\times\ldots\times\mathbb Z_{N_k}$. See \cite{tzy}.

The complex Hadamard matrices are usually regarded modulo equivalence:

\begin{definition}
Two complex Hadamard matrices $H,K\in M_N(\mathbb C)$ are called equivalent, and we write $H\sim K$, if one can pass from one to the other by permuting the rows and columns, or by multiplying these rows and columns by numbers in $\mathbb T$.
\end{definition}

As explained in the introduction, each complex Hadamard matrix produces a subfactor $M\subset R$ of the Murray-von Neumann hyperfinite factor $R$, having index $[R:M]=N$, which can be understood in terms of quantum groups. Indeed, let us call ``magic'' any square matrix $u=(u_{ij})$ whose entries are projections ($p=p^2=p^*$), summing up to $1$ on each row and column. We have then the following key definition, due to Wang \cite{wan}:

\begin{definition}
$C(S_N^+)$ is the universal $C^*$-algebra generated by the entries of a $N\times N$ magic matrix $u=(u_{ij})$, with comultiplication, counit and antipode maps defined on the standard generators by $\Delta(u_{ij})=\sum_ku_{ik}\otimes u_{kj}$, $\varepsilon(u_{ij})=\delta_{ij}$ and $S(u_{ij})=u_{ji}$.
\end{definition}

As explained in \cite{wan}, this algebra satisfies Woronowicz's axioms in \cite{wo1}, and so $S_N^+$ is a compact quantum group, called quantum permutation group. Since the functions $v_{ij}:S_N\to\mathbb C$ given by $v_{ij}(\sigma)=\delta_{i\sigma(j)}$ form a magic matrix, we have a quotient map $C(S_N^+)\to C(S_N)$, which corresponds to an embedding $S_N\subset S_N^+$. This embedding is an isomorphism at $N=1,2,3$, but not at $N\geq4$, where $S_N^+$ is not finite. See \cite{wan}.

The link with the Hadamard matrices comes from:

\begin{definition}
Associated to an Hadamard matrix $H\in M_N(\mathbb T)$ is the minimal quantum group $G\subset S_N^+$ producing a factorization of type
$$\xymatrix{C(S_N^+)\ar[rr]^\pi\ar[rd]&&M_N(\mathbb C)\\&C(G)\ar[ur]_\rho&}$$
where $\pi:u_{ij}\to P_{ij}=Proj(H_i/H_j)$, where $H_1,\ldots,H_N\in\mathbb T^N$ are the rows of $H$.
\end{definition}

Here the fact that $\pi$ is indeed well-defined follows from the fact that $P=(P_{ij})$ is magic, which comes from the fact that the rows of $H$ are pairwise orthogonal. As for the existence and uniqueness of the quantum group $G\subset S_N^+$ as in the statement, this comes from Hopf algebra theory, by dividing $C(S_N^+)$ by a suitable ideal. See \cite{bbi}.

At the level of examples, it is known that the Fourier matrix $F_G$ produces the group $G$ itself. In general, the computation of $G$ is a quite difficult question. See \cite{bfs}.

At a theoretical level, it is known that the above-mentioned subfactor $M\subset R$ associated to $H$ appears as a fixed point subfactor associated to $G$. See \cite{ban}.

In what follows we will rather use a representation-theoretic formulation of this latter result. Let $u=(u_{ij})$ be the fundamental representation of $G$.

\begin{definition}
We let $\mu\in\mathcal P[0,N]$ be the law of variable $\chi=\sum_iu_{ii}$, with respect to the Haar integration functional of $C(G)$.
\end{definition}

Note that the main character $\chi=\sum_iu_{ii}$ being a sum of $N$ projections, we have the operator-theoretic formula $0\leq\chi\leq N$, and so $supp(\mu)\subset[0,N]$, as stated above.

Observe also that the moments of $\mu$ are integers, because we have the following computation, based on Woronowicz's general Peter-Weyl type results in \cite{wo1}:
$$\int_0^Nx^kd\mu(x)=\int_GTr(u)^k=\int_GTr(u^{\otimes k})=\dim(Fix(u^{\otimes k}))$$

The above moments, or rather the fixed point spaces appearing on the right, can be computed by using the following fundamental result, from \cite{bbi}:

\begin{theorem}
We have an equality of complex vector spaces
$$Fix(u^{\otimes k})=Fix(P^{\otimes k})$$
where for $X\in M_N(A)$ we set $X^{\otimes k}=(X_{i_1j_1}\ldots X_{i_kj_k})_{i_1\ldots i_k,j_1\ldots j_k}$.
\end{theorem}

Now back to the subfactor problematics, it is known from \cite{jsu} that the planar algebra associated to $H$ is given by $P_k=Fix(P^{\otimes k})$. Thus, Theorem 1.5 tells us that the Poincar\'e series $f(z)=\sum_{k=0}^\infty\dim(P_k)z^k$ is nothing but the Stieltjes transform of $\mu$:
$$f(z)=\int_G\frac{1}{1-z\chi}$$

Summarizing, modulo some standard correspondences, the main subfactor problem regarding $H$ consists in computing the spectral measure $\mu$ in Definition 1.4.

\section{Finiteness, duality}

We discuss in this section a key issue, namely the formulation of the duality between the quantum permutation groups associated to the matrices $H,\overline{H},H^t,H^*$. Our claim is that the general scheme for this duality is, roughly speaking, as follows:
\vskip-10mm
$$\xymatrix{H\ar@{-}[d]\ar@{-}[r]&H^t\\ \overline{H}&H^*\ar@{-}[l]\ar@{-}[u]}\quad\begin{matrix}\\ \\ \\ \implies\end{matrix}\quad\xymatrix{G\ar@{-}[d]\ar@{-}[r]&\widehat{G}\\ G^\sigma&\widehat{G}^\sigma\ar@{-}[l]\ar@{-}[u]}$$

More precisely, this scheme fully works when the quantum groups are finite. In the general case the situation is more complicated, as explained in \cite{ban}. 

The results in \cite{ban}, written some time ago, in the general context of vertex models, and without using the Hopf image formalism in \cite{bbi}, are in fact not very enlightening in the Hadamard matrix case. We will present below an updated approach. First, we have:

\begin{proposition}
The matrices $P=(P_{ij})$ for $H,\overline{H},H^t,H^*$ are related by:
\vskip-8mm
$$\xymatrix{H\ar@{-}[d]\ar@{-}[r]&H^t\\ \overline{H}&H^*\ar@{-}[l]\ar@{-}[u]}\quad\begin{matrix}\\ \\ \\ \implies\end{matrix}\quad\xymatrix{(P_{ij})_{kl}\ar@{-}[d]\ar@{-}[r]&(P_{kl})_{ij}\\ (P_{ji})_{kl}&(P_{kl})_{ji}\ar@{-}[l]\ar@{-}[u]}$$
In addition, we have the formula $(P_{ij})_{kl}=(P_{ji})_{lk}$.
\end{proposition}

\begin{proof}
The magic matrix associated to $H$ is given by $P_{ij}=Proj(H_i/H_j)$. Now since $H\to\overline{H}$ transforms $H_i/H_j\to H_j/H_i$, we conclude that the magic matrices $P^H,P^{\overline{H}}$ associated to $H,\overline{H}$ are related by the formula $P_{ij}^H=P_{ji}^{\overline{H}}$, as stated above. 

In matrix notation, the formula for the matrix $P^H$ is as follows: 
$$(P^H_{ij})_{kl}=\frac{1}{N}\cdot\frac{H_{ik}H_{jl}}{H_{il}H_{jk}}$$

Now by replacing $H\to H^t$, we obtain the following formula:
$$(P^{H^t}_{ij})_{kl}=\frac{1}{N}\cdot\frac{H_{ki}H_{lj}}{H_{li}H_{kj}}=(P^H_{kl})_{ij}$$

Finally, the last assertion is clear from the above formula of $P^H$.
\end{proof}

Let us compute now Hopf images. First, regarding the operation $H\to\overline{H}$, we have:

\begin{proposition}
The quantum groups associated to $H,\overline{H}$ are related by
$$G_{\overline{H}}=G_H^\sigma$$
where the Hopf algebra $C(G^\sigma)$ is $C(G)$ with comultiplication $\Sigma\Delta$, where $\Sigma$ is the flip.
\end{proposition}

\begin{proof}
Our claim is that, starting from a factorization for $H$ as in Definition 1.3 above, we can construct a factorization for $\overline{H}$, as follows:
\vskip-10mm
$$\xymatrix{u_{ij}\ar[rr]\ar[rd]&&P_{ij}\\&v_{ij}\in C(G)\ar[ur]&}\quad\begin{matrix}\\ \\ \\ \implies\end{matrix}\quad
\xymatrix{u_{ij}\ar[rr]\ar[rd]&&P_{ji}\\&v_{ji}\in C(G^\sigma)\ar[ur]&}$$

Indeed, observe first that since $v_{ij}\in C(G)$ are the coefficients of a corepresentation, then so are the elements $v_{ji}\in C(G^\sigma)$. Thus, in order to produce the factorization on the right, it is enough to take the diagram on the left, and compose at top left with the canonical map $C(S_N^+)\to C(S_N^{+\sigma})$ given by $u_{ij}\to u_{ji}$, and we are done.
\end{proof}

Let us investigate now the operation $H\to H^t$. We use the notion of dual of a finite quantum group, see e.g. \cite{wo1}. The result here is as follows: 

\begin{theorem}
The quantum groups associated to $H,H^t$ are related by usual duality,
$$G_{H^t}=\widehat{G}_H$$
provided that the quantum group $G_H$ is finite. 
\end{theorem}

\begin{proof}
Our claim is that, starting from a factorization for $H$ as in Definition 1.3 above, we can construct a factorization for $H^t$, as follows:
\vskip-10mm
$$\xymatrix{C(S_N^+)\ar[rr]^{\pi_H}\ar[rd]&&M_N(\mathbb C)\\&C(G)\ar[ur]_\rho&}\quad\begin{matrix}\\ \\ \\ \implies\end{matrix}\quad
\xymatrix{C(S_N^+)\ar[rr]^{\pi_{H^t}}\ar[rd]&&M_N(\mathbb C)\\&C(G)^*\ar[ur]_\eta&}$$

More precisely, having a factorization as the one on the left, let us set:
\begin{eqnarray*}
\eta(\varphi)&=&(\varphi(v_{kl}))_{kl}\\
w_{kl}(x)&=&(\rho(x))_{kl}
\end{eqnarray*}

Our claim is that $\eta$ is a representation, $w$ is a corepresentation, and the factorization on the right holds indeed. Let us first check that $\eta$ is a representation:
\begin{eqnarray*}
\eta(\varphi\psi)&=&(\phi\psi(v_{kl}))_{kl}=((\varphi\otimes\psi)\Delta(v_{kl}))_{kl}=(\sum_a\varphi(v_{ka})\psi(v_{al}))_{kl}=\eta(\varphi)\eta(\psi)\\
\eta(\varepsilon)&=&(\varepsilon(v_{kl}))_{kl}=(\delta_{kl})_{kl}=1\\
\eta(\varphi^*)&=&(\varphi^*(v_{kl}))_{kl}=(\overline{\varphi(S(v_{kl}^*))})_{kl}=(\overline{\varphi(v_{lk})})_{kl}=\eta(\varphi)^*
\end{eqnarray*}

Let us check now the fact that $w$ is a corepresentation:
\begin{eqnarray*}
(\Delta w_{kl})(x\otimes y)
&=&w_{kl}(xy)=\rho(xy)_{kl}=\sum_i\rho(x)_{ki}\rho(y)_{il}\\
&=&\sum_iw_{ki}(x)w_{il}(y)=(\sum_iw_{ki}\otimes w_{il})(x\otimes y)\\
\varepsilon(w_{kl})&=&w_{kl}(1)=1_{kl}=\delta_{kl}
\end{eqnarray*}

We check now the fact that the above diagram commutes on the generators $u_{ij}$:
$$\eta(w_{ab})=(w_{ab}(v_{kl}))_{kl}=(\rho(v_{kl})_{ab})_{kl}=((P^H_{kl})_{ab})_{kl}=((P^{H^t}_{ab})_{kl})_{kl}=P^{H^t}_{ab}$$

It remains to prove that $w$ is magic. We have the following formula:
\begin{eqnarray*}
w_{a_0a_p}(v_{i_1j_1}\ldots v_{i_pj_p})
&=&(\Delta^{(p-1)}w_{a_0a_p})(v_{i_1j_1}\otimes\ldots\otimes v_{i_pj_p})\\
&=&\sum_{a_1\ldots a_{p-1}}w_{a_0a_1}(v_{i_1j_1})\ldots w_{a_{p-1}a_p}(v_{i_pj_p})\\
&=&\frac{1}{N^p}\sum_{a_1\ldots a_{p-1}}\frac{H_{i_1a_0}H_{j_1a_1}}{H_{i_1a_1}H_{j_1a_0}}\ldots\ldots\frac{H_{i_pa_{p-1}}H_{j_pa_p}}{H_{i_pa_p}H_{j_pa_{p-1}}}
\end{eqnarray*}

In order to check that each $w_{ab}$ is an idempotent, observe that we have:
\begin{eqnarray*}
w_{a_0a_p}^2(v_{i_1j_1}\ldots v_{i_pj_p})
&=&(w_{a_0a_p}\otimes w_{a_0a_p})\sum_{k_1\ldots k_p}v_{i_1k_1}\ldots v_{i_pk_p}\otimes v_{k_1j_1}\ldots v_{k_pj_p}\\
&=&\frac{1}{N^{2p}}\sum_{k_1\ldots k_p}\sum_{a_1\ldots a_{p-1}}\sum_{\alpha_1\ldots \alpha_{p-1}}\frac{H_{i_1a_0}H_{k_1a_1}}{H_{i_1a_1}H_{k_1a_0}}\ldots\ldots\frac{H_{i_pa_{p-1}}H_{k_pa_p}}{H_{i_pa_p}H_{k_pa_{p-1}}}\\
&&\frac{H_{k_1a_0}H_{j_1\alpha_1}}{H_{k_1\alpha_1}H_{j_1a_0}}\ldots\ldots\frac{H_{k_p\alpha_{p-1}}H_{j_pa_p}}{H_{k_pa_p}H_{j_p\alpha_{p-1}}}
\end{eqnarray*}

The point now is that when summing over $k_1$ we obtain $N\delta_{a_1\alpha_1}$, then when summing over $k_2$ we obtain $N\delta_{a_2\alpha_2}$, and so on up to summing over $k_{p-1}$, where we obtain $N\delta_{a_{p-1}\alpha_{p-1}}$. Thus, after performing all these summations, what we are left with is:
\begin{eqnarray*}
w_{a_0a_p}^2(v_{i_1j_1}\ldots v_{i_pj_p})
&=&\frac{1}{N^{p+1}}\sum_{k_p}\sum_{a_1\ldots a_{p-1}}\frac{H_{i_1a_0}H_{j_1a_1}}{H_{i_1a_1}H_{j_1a_0}}\ldots\ldots\frac{H_{i_pa_{p-1}}H_{k_pa_p}}{H_{i_pa_p}H_{k_pa_{p-1}}}\cdot\frac{H_{k_pa_{p-1}}H_{j_pa_p}}{H_{k_pa_p}H_{j_pa_{p-1}}}\\
&=&\frac{1}{N^{p+1}}\sum_{k_p}\sum_{a_1\ldots a_{p-1}}\frac{H_{i_1a_0}H_{j_1a_1}}{H_{i_1a_1}H_{j_1a_0}}\ldots\ldots\frac{H_{i_pa_{p-1}}H_{j_pa_p}}{H_{i_pa_p}H_{j_pa_{p-1}}}\\
&=&\frac{1}{N^p}\sum_{a_1\ldots a_{p-1}}\frac{H_{i_1a_0}H_{j_1a_1}}{H_{i_1a_1}H_{j_1a_0}}\ldots\ldots\frac{H_{i_pa_{p-1}}H_{j_pa_p}}{H_{i_pa_p}H_{j_pa_{p-1}}}\\
&=&w_{a_0a_p}(v_{i_1j_1}\ldots v_{i_pj_p})
\end{eqnarray*}

Regarding now the involutivity, the check here is simply:
\begin{eqnarray*}
w_{a_0a_p}^*(v_{i_1j_1}\ldots v_{i_pj_p})
&=&\overline{w_{a_0a_p}(S(v_{i_pj_p}\ldots v_{i_1j_1}))}\\
&=&\overline{w_{a_0a_p}(v_{j_1i_1}\ldots v_{j_pi_p})}\\
&=&w_{a_0a_p}^*(v_{i_1j_1}\ldots v_{i_pj_p})
\end{eqnarray*}

Finally, for checking the first ``sum 1'' condition, observe that we have:
$$\sum_{a_0}w_{a_0a_p}(v_{i_1j_1}\ldots v_{i_pj_p})
=\frac{1}{N^p}\sum_{a_0\ldots a_{p-1}}\frac{H_{i_1a_0}H_{j_1a_1}}{H_{i_1a_1}H_{j_1a_0}}\ldots\ldots\frac{H_{i_pa_{p-1}}H_{j_pa_p}}{H_{i_pa_p}H_{j_pa_{p-1}}}$$

The point now is that when summing over $a_0$ we obtain $N\delta_{i_1j_1}$, then when summing over $a_1$ we obtain $N\delta_{i_2j_2}$, and so on up to summing over $a_{p-1}$, where we obtain $N\delta_{i_pj_p}$. Thus, after performing all these summations, what we are left with is:
$$\sum_{a_0}w_{a_0a_p}(v_{i_1j_1}\ldots v_{i_pj_p})=\delta_{i_1j_1}\ldots\delta_{i_pj_p}=\varepsilon(v_{i_1j_1}\ldots v_{i_pj_p})$$

The proof of the other ``sum 1'' condition is similar, and this finishes the proof.
\end{proof}

\section{The truncation procedure}

Let us go back now to the factorization in Definition 1.3. Regarding the Haar functional of the quantum group $G$, we have the following key result, from \cite{bfs}:

\begin{proposition}
We have the Ces\`aro limiting formula
$$\int_G=\lim_{k\to\infty}\frac{1}{k}\sum_{r=1}^k\int_G^r$$
where the functionals at right are by definition given by $\int_G^r=(tr\circ\rho)^{*r}$.
\end{proposition}

Regarding the functionals $\int_G^r$, their evaluation is a linear algebra problem. Several formulations of the problem were proposed in \cite{bfs}, and we will use here the following formula, which appears there, but in a somewhat technical form:

\begin{proposition}
The functionals $\int_G^r=(tr\circ\rho)^{*r}$ are given by
$$\int_G^ru_{a_1b_1}\ldots u_{a_pb_p}=(T_p^r)_{a_1\ldots a_p,b_1\ldots b_p}$$
where $(T_p)_{i_1\ldots i_p,j_1\ldots j_p}=tr(P_{i_1j_1}\ldots P_{i_pj_p})$, with $P_{ij}=Proj(H_i/H_j)$.
\end{proposition}

\begin{proof}
With $a_s=i_s^0$ and $b_s=i_s^{r+1}$, we have the following computation:
\begin{eqnarray*}
\int_G^ru_{a_1b_1}\ldots u_{a_pb_p}
&=&(tr\circ\rho)^{\otimes r}\Delta^{(r)}(u_{i_1^0i_1^{r+1}}\ldots u_{i_p^0i_p^{r+1}})\\
&=&(tr\circ\rho)^{\otimes r}\sum_{i_1^1\ldots i_p^r}u_{i_1^0i_1^1}\ldots u_{i_p^0i_p^1}\otimes\ldots\ldots\otimes u_{i_1^ru_1^{r+1}}\ldots u_{i_p^ri_p^{r+1}}\\
&=&tr^{\otimes r}\sum_{i_1^1\ldots i_p^r}P_{i_1^0i_1^1}\ldots P_{i_p^0i_p^1}\otimes\ldots\ldots\otimes P_{i_1^ri_1^{r+1}}\ldots P_{i_p^ri_p^{r+1}}
\end{eqnarray*}

On the other hand, we have as well the following computation:
\begin{eqnarray*}
(T_p^r)_{a_1\ldots a_p,b_1\ldots b_p}
&=&\sum_{i_1^1\ldots i_p^r}(T_p)_{i_1^0\ldots i_p^0,i_1^1\ldots i_p^1}\ldots\ldots (T_p)_{i_1^r\ldots i_p^r,i_1^{r+1}\ldots i_p^{r+1}}\\
&=&\sum_{i_1^1\ldots i_p^r}tr(P_{i_1^0i_1^1}\ldots P_{i_p^0i_p^1})\ldots\ldots tr(P_{i_1^ri_1^{r+1}}\ldots P_{i_p^ri_p^{r+1}})\\
&=&tr^{\otimes r}\sum_{i_1^1\ldots i_p^r}P_{i_1^0i_1^1}\ldots P_{i_p^0i_p^1}\otimes\ldots\ldots\otimes P_{i_1^ri_1^{r+1}}\ldots P_{i_p^ri_p^{r+1}}
\end{eqnarray*}

Thus we have obtained the formula in the statement, and we are done.
\end{proof}

We can now define the truncations of $\mu$, as follows:

\begin{proposition}
Let $\mu^r$ be the law of $\chi$ with respect to $\int_G^r=(tr\circ\rho)^{*r}$.
\begin{enumerate}
\item $\mu^r$ is a probability measure on $[0,N]$.

\item We have the formula $\mu=\lim_{k\to\infty}\frac{1}{k}\sum_{r=1}^k\mu^r$.

\item The moments of $\mu^r$ are the numbers $c_p^r=Tr(T_p^r)$.
\end{enumerate}
\end{proposition}

\begin{proof}
(1) The fact that $\mu^r$ is indeed a probability measure follows from the fact that the linear form $(tr\circ\rho)^{*r}:C(G)\to\mathbb C$ is a positive unital trace, and the assertion on the support comes from the fact that the main character $\chi$ is a sum of $N$ projections.

(2) This follows from Proposition 3.1, i.e. from the main result in \cite{bfs}.

(3) This follows from Proposition 3.2 above, by summing over $a_i=b_i$.
\end{proof}

Let us recall now that associated to a complex Hadamard matrix $H\in M_N(\mathbb C)$ is its profile matrix, given by:
$$Q_{ab,cd}=\frac{1}{N}\left\langle\frac{H_a}{H_b},\frac{H_c}{H_d}\right\rangle=\frac{1}{N}\sum_i\frac{H_{ia}H_{id}}{H_{ib}H_{ic}}$$

With this notation, we have the following result:

\begin{proposition}
The measures $\mu^r$ have the following properties:
\begin{enumerate}
\item $\mu^0=\delta_N$.

\item $\mu^1=(1-\frac{1}{N})\delta_0+\frac{1}{N}\delta_N$.

\item $\mu^2=law(S)$, where $S_{ab,cd}=|Q_{ab,cd}|^2$.

\item For a Fourier matrix $F_G$ we have $\mu^1=\mu^2=\ldots=\mu$.
\end{enumerate}
\end{proposition}

\begin{proof}
We use the formula $c_p^r=Tr(T_p^r)$ from Proposition 3.3 (3) above.

(1) At $r=0$ we have $c_p^0=Tr(T_p^0)=Tr(Id_{N^p})=N^p$, so $\mu^0=\delta_N$.

(2) At $r=1$, if we denote by $J$ the flat matrix $(1/N)_{ij}$, we have indeed: 
$$c_p^1=Tr(T_p)=\sum_{i_1\ldots i_p}tr(P_{i_1i_1}\ldots P_{i_pi_p})=\sum_{i_1\ldots i_p}tr(J^p)=\sum_{i_1\ldots i_p}tr(J)=N^{p-1}$$

(3) This can be checked directly, and is also a consequence of Theorem 3.5 below.

(4) For a Fourier matrix the representation $\rho$ producing the factorization in Definition 1.3 is well-known to be faithful, and this gives the result.
\end{proof}

In the general case now, we have the following result:

\begin{theorem}
We have $\mu^r=law(X)$, where
$$X_{a_1\ldots a_r,b_1\ldots b_r}=Q_{a_1b_1,a_2b_2}\ldots Q_{a_rb_r,a_1b_1}$$
where $Q$ denotes as usual the profile matrix.
\end{theorem}

\begin{proof}
We compute the moments of $\mu^r$. We first have:
\begin{eqnarray*}
c_p^r
&=&Tr(T_p^r)=\sum_{i^1\ldots i^r}(T_p)_{i^1i^2}\ldots(T_p)_{i^ri^1}\\
&=&\sum_{i_1^1\ldots i_p^r}(T_p)_{i_1^1\ldots i_p^1,i_1^2\ldots i_p^2}\ldots\ldots(T_p)_{i_1^r\ldots i_p^r,i_1^1\ldots i_p^1}\\
&=&\sum_{i_1^1\ldots i_p^r}tr(P_{i_1^1i_1^2}\ldots P_{i_p^1i_p^2})\ldots\ldots tr(P_{i_1^ri_1^1}\ldots P_{i_p^ri_p^1})
\end{eqnarray*}

In terms of $H$, we obtain the following formula:
\begin{eqnarray*}
c_p^r
&=&\frac{1}{N^r}\sum_{i_1^1\ldots i_p^r}\sum_{a_1^1\ldots a_p^r}(P_{i_1^1i_1^2})_{a_1^1a_2^1}\ldots(P_{i_p^1i_p^2})_{a_p^1a_1^1}\ldots\ldots (P_{i_1^ri_1^1})_{a_1^ra_2^r}\ldots(P_{i_p^ri_p^1})_{a_p^ra_1^r}\\
&=&\frac{1}{N^{(p+1)r}}\sum_{i_1^1\ldots i_p^r}\sum_{a_1^1\ldots a_p^r}
\frac{H_{i_1^1a_1^1}H_{i_1^2a_2^1}}{H_{i_1^1a_2^1}H_{i_1^2a_1^1}}\ldots\frac{H_{i_p^1a_p^1}H_{i_p^2a_1^1}}{H_{i_p^1a_1^1}H_{i_p^2a_p^1}}\ldots\ldots\frac{H_{i_1^ra_1^r}H_{i_1^1a_2^r}}{H_{i_1^ra_2^r}H_{i_1^1a_1^r}}\ldots\frac{H_{i_p^ra_p^r}H_{i_p^1a_1^r}}{H_{i_p^ra_1^r}H_{i_p^1a_p^r}}
\end{eqnarray*}

Now by changing the order of the summation, we obtain:
\begin{eqnarray*}
c_p^r=\frac{1}{N^{(p+1)r}}\sum_{a_1^1\ldots a_p^r}
&&\sum_{i_1^1}\frac{H_{i_1^1a_1^1}H_{i_1^1a_2^r}}{H_{i_1^1a_2^1}H_{i_1^1a_1^r}}\ldots\ldots\sum_{i_1^r}\frac{H_{i_1^ra_2^{r-1}}H_{i_1^ra_1^r}}{H_{i_1^ra_1^{r-1}}H_{i_1^ra_2^r}}\\
&&\ldots\\
&&\sum_{i_p^1}\frac{H_{i_p^1a_p^1}H_{i_p^1a_1^r}}{H_{i_p^1a_1^1}H_{i_p^1a_p^r}}\ldots\ldots\sum_{i_p^r}\frac{H_{i_p^ra_1^{r-1}}H_{i_p^ra_p^r}}{H_{i_p^ra_p^{r-1}}H_{i_p^ra_1^r}}
\end{eqnarray*}

In terms of $Q$, and then of the matrix $X$ in the statement, we get:
\begin{eqnarray*}
c_p^r
&=&\frac{1}{N^r}\sum_{a_1^1\ldots a_p^r}(Q_{a_1^1a_2^1,a_1^ra_2^r}\ldots Q_{a_1^ra_2^r,a_1^{r-1}a_2^{r-1}}) \ldots\ldots(Q_{a_p^1a_1^1,a_p^ra_1^r}\ldots Q_{a_p^ra_1^r,a_p^{r-1}a_1^{r-1}})\\
&=&\frac{1}{N^r}\sum_{a_1^1\ldots a_p^r}X_{a_1^1\ldots a_1^r,a_2^1\ldots a_2^r}\ldots\ldots X_{a_p^1\ldots a_p^r,a_1^1\ldots a_1^r}\\
&=&\frac{1}{N^r}Tr(X^p)=tr(X^p)
\end{eqnarray*}

But this gives the formula in the statement, and we are done.
\end{proof}

Observe that the above result covers the previous computations of $\mu^0,\mu^1,\mu^2$, and in particular the formula of $\mu^2$ in Proposition 3.4 (3). Indeed, at $r=2$ we have:
$$X_{ab,cd}=Q_{ac,bd}Q_{bd,ac}=Q_{ab,cd}\overline{Q_{ab,cd}}=|Q_{ab,cd}|^2$$

We will discuss in the next section some further interpretations of $\mu^r$.

\section{Basic properties, examples}

Let us first take a closer look at the matrices $X$ appearing in Theorem 3.5. These are in fact Gram matrices, of certain norm one vectors:

\begin{proposition}
We have $\mu^r=law(X)$, with $X_{a_1\ldots a_r,b_1\ldots b_r}=<\xi_{a_1\ldots a_r},\xi_{b_1\ldots b_r}>$, where: 
$$\xi_{a_1\ldots a_r}=\frac{1}{\sqrt{N}}\cdot\frac{H_{a_1}}{H_{a_2}}\otimes\ldots\otimes\frac{1}{\sqrt{N}}\cdot\frac{H_{a_r}}{H_{a_1}}$$
In addition, these vectors $\xi_{a_1\ldots a_r}$ are all of norm one.
\end{proposition}

\begin{proof}
The first assertion follows from the following computation:
\begin{eqnarray*}
X_{a_1\ldots a_r,b_1\ldots b_r}
&=&\frac{1}{N^r}\left\langle\frac{H_{a_1}}{H_{b_1}},\frac{H_{a_2}}{H_{b_2}}\right\rangle\ldots\left\langle\frac{H_{a_r}}{H_{b_r}},\frac{H_{a_1}}{H_{b_1}}\right\rangle\\
&=&\frac{1}{N^r}\left\langle\frac{H_{a_1}}{H_{a_2}},\frac{H_{b_1}}{H_{b_2}}\right\rangle\ldots\left\langle\frac{H_{a_r}}{H_{a_1}},\frac{H_{b_r}}{H_{b_1}}\right\rangle\\
&=&\frac{1}{N^r}\left\langle\frac{H_{a_1}}{H_{a_2}}\otimes\ldots\otimes\frac{H_{a_r}}{H_{a_1}},\frac{H_{b_1}}{H_{b_2}}\otimes\ldots\otimes\frac{H_{b_r}}{H_{b_1}}\right\rangle
\end{eqnarray*}

As for the second assertion, this is clear from the formula of $\xi_{a_1\ldots a_r}$.
\end{proof}

At the level of concrete examples now, we first have:

\begin{proposition}
For a Fourier matrix $H=F_G$ we have:
\begin{enumerate}
\item $Q_{ab,cd}=\delta_{a+d,b+c}$.

\item $X_{a_1\ldots a_r,b_1\ldots b_r}=\delta_{a_1-b_1,\ldots,a_r-b_r}$.

\item $X^2=NX$, so $X/N$ is a projection.
\end{enumerate}
\end{proposition}

\begin{proof}
We use the formulae $H_{ij}H_{ik}=H_{i,j+k}$, $\overline{H}_{ij}=H_{i,-j}$ and $\sum_iH_{ij}=N\delta_{j0}$.

(1) We have indeed the following computation:
$$Q_{ab,cd}=\frac{1}{N}\sum_iH_{i,a+d-b-c}=\delta_{a+d,b+c}$$

(2) This follows from the following computation:
\begin{eqnarray*}
X_{a_1\ldots a_r,b_1\ldots b_r}
&=&\delta_{a_1+b_2,b_1+a_2}\ldots\delta_{a_r+b_1,b_r+a_1}\\
&=&\delta_{a_1-b_1,a_2-b_2}\ldots\delta_{a_r-b-r,a_1-b_1}\\
&=&\delta_{a_1-b_1,\ldots,a_r-b_r}
\end{eqnarray*}

(3) By using the formula in (2) above, we obtain:
\begin{eqnarray*}
(X^2)_{a_1\ldots a_r,b_1\ldots b_r}
&=&\sum_{c_1\ldots c_r}X_{a_1\ldots a_r,c_1\ldots c_r}X_{c_1\ldots c_r,b_1\ldots b_r}\\
&=&\sum_{c_1\ldots c_r}\delta_{a_1-c_1,\ldots,a_r-c_r}\delta_{c_1-b_1,\ldots,c_r-b_r}\\
&=&N\delta_{a_1-b_1,\ldots,a_r-b_r}=NX_{a_1\ldots a_r,b_1\ldots b_r}
\end{eqnarray*}

Thus $(X/N)^2=X/N$, and since $X/N$ is as well self-adjoint, it is a projection.
\end{proof}

Another situation which is elementary is the tensor product one:

\begin{proposition}
Let $L=H\otimes K$.
\begin{enumerate}
\item $Q^L_{iajb,kcld}=Q^H_{ij,kl}Q^K_{ab,cd}$.

\item $X^L_{i_1a_1\ldots i_ra_r,j_1b_1\ldots j_rb_r}=X^H_{i_1\ldots i_r,j_1\ldots j_r}X^K_{a_1\ldots a_r,b_1\ldots b_r}$.

\item $\mu_L^r=\mu_H^r*\mu_K^r$, for any $r\geq0$.
\end{enumerate}
\end{proposition}

\begin{proof}
(1) This follows from the following computation:
\begin{eqnarray*}
Q^L_{iajb,kcld}
&=&\frac{1}{NM}\sum_{me}\frac{L_{me,ia}L_{me,ld}}{L_{me,kc}L_{me,jb}}
=\frac{1}{NM}\sum_{me}\frac{H_{mi}K_{ea}H_{ml}K_{ld}}{H_{mk}K_{ec}H_{mj}K_{eb}}\\
&=&\frac{1}{N}\sum_m\frac{H_{mi}H_{ml}}{H_{mk}H_{mj}}\cdot\frac{1}{M}\sum_e\frac{K_{ea}K_{ed}}{K_{ec}K_{eb}}=Q^H_{ij,kl}Q^K_{ab,cd}
\end{eqnarray*}

(2) This follows from (2) above, because we have:
\begin{eqnarray*}
X^L_{i_1a_1\ldots i_ra_r,j_1b_1\ldots j_rb_r}
&=&Q^L_{i_1a_1j_1b_1,1_2a_2j_2b_2}\ldots Q^L_{i_ra_rj_rb_r,i_1a_1j_1b_1}\\
&=&Q^H_{i_1j_1,i_2j_2}Q^K_{a_1b_1,a_2b_2}\ldots Q^H_{i_rj_r,i_1j_1}Q^K_{a_rb_r,a_1b_1}\\
&=&X^H_{i_1\ldots i_r,j_1\ldots j_r}X^K_{a_1\ldots a_r,b_1\ldots b_r}
\end{eqnarray*}

(3) This follows from (3) above, which tells us that, modulo certain standard indentifications, we have $X^L=X^H\otimes X^K$.
\end{proof}

We will be back in section 5 below to the study of concrete examples. Now let us discuss some general duality issues. We have here:

\begin{theorem}
We have the moment/truncation duality formula
$$\int_{G_H}^r\left(\frac{\chi}{N}\right)^p=\int_{G_{H^t}}^p\left(\frac{\chi}{N}\right)^r$$
where $G_H,G_{H^t}$ are the quantum groups associated to $H,H^t$.
\end{theorem}

\begin{proof}
We use the following formula, from the proof of Theorem 3.5:
$$c_p^r=\frac{1}{N^{(p+1)r}}\sum_{i_1^1\ldots i_p^r}\sum_{a_1^1\ldots a_p^r}
\frac{H_{i_1^1a_1^1}H_{i_1^2a_2^1}}{H_{i_1^1a_2^1}H_{i_1^2a_1^1}}\ldots\frac{H_{i_p^1a_p^1}H_{i_p^2a_1^1}}{H_{i_p^1a_1^1}H_{i_p^2a_p^1}}\ldots\ldots\frac{H_{i_1^ra_1^r}H_{i_1^1a_2^r}}{H_{i_1^ra_2^r}H_{i_1^1a_1^r}}\ldots\frac{H_{i_p^ra_p^r}H_{i_p^1a_1^r}}{H_{i_p^ra_1^r}H_{i_p^1a_p^r}}$$

By interchanging $p\leftrightarrow r$, and by transposing as well all the summation indices, according to the rules $i_x^y\to i_y^x$ and $a_x^y\to a_y^x$, we obtain the following formula:
$$c_r^p=\frac{1}{N^{(r+1)p}}\sum_{i_1^1\ldots i_p^r}\sum_{a_1^1\ldots a_p^r}
\frac{H_{i_1^1a_1^1}H_{i_2^1a_1^2}}{H_{i_1^1a_1^2}H_{i_2^1a_1^1}}\ldots\frac{H_{i_1^ra_1^r}H_{i_2^ra_1^1}}{H_{i_1^ra_1^1}H_{i_2^ra_1^r}}\ldots\ldots\frac{H_{i_p^1a_p^1}H_{i_1^1a_p^2}}{H_{i_p^1a_p^2}H_{i_1^1a_p^1}}\ldots\frac{H_{i_p^ra_p^r}H_{i_1^ra_p^1}}{H_{i_p^ra_p^1}H_{i_1^ra_p^r}}$$

Now by interchaging all the summation indices, $i_x^y\leftrightarrow a_x^y$, we obtain:
$$c_r^p=\frac{1}{N^{(r+1)p}}\sum_{i_1^1\ldots i_p^r}\sum_{a_1^1\ldots a_p^r}
\frac{H_{a_1^1i_1^1}H_{a_2^1i_1^2}}{H_{a_1^1i_1^2}H_{a_2^1i_1^1}}\ldots\frac{H_{a_1^ri_1^r}H_{a_2^ri_1^1}}{H_{a_1^ri_1^1}H_{a_2^ri_1^r}}\ldots\ldots\frac{H_{a_p^1i_p^1}H_{a_1^1i_p^2}}{H_{a_p^1i_p^2}H_{a_1^1i_p^1}}\ldots\frac{H_{a_p^ri_p^r}H_{a_1^ri_p^1}}{H_{a_p^ri_p^1}H_{a_1^ri_p^r}}$$

With $H\to H^t$, we obtain the following formula, this time for $H^t$:
$$c_r^p=\frac{1}{N^{(r+1)p}}\sum_{i_1^1\ldots i_p^r}\sum_{a_1^1\ldots a_p^r}
\frac{H_{i_1^1a_1^1}H_{i_1^2a_2^1}}{H_{i_1^2a_1^1}H_{i_1^1a_2^1}}\ldots\frac{H_{i_1^ra_1^r}H_{i_1^1a_2^r}}{H_{i_1^1a_1^r}H_{i_1^ra_2^r}}\ldots\ldots\frac{H_{i_p^1a_p^1}H_{i_p^2a_1^1}}{H_{i_p^2a_p^1}H_{i_p^1a_1^1}}\ldots\frac{H_{i_p^ra_p^r}H_{i_p^1a_1^r}}{H_{i_p^1a_p^r}H_{i_p^ra_1^r}}$$

The point now is that, modulo a permutation of terms, the quantity on the right is exactly the one as in the above formula of $c_p^r$. Thus, if we denote by $\alpha$ this quantity:
$$c_p^r(H)=\frac{\alpha}{N^{(p+1)r}},\qquad c_r^p(H^t)=\frac{\alpha}{N^{(r+1)p}}$$

Thus we have $N^rc_p^r(H)=N^pc_r^p(H^t)$, and by dividing by $N^{p+r}$, we obtain:
$$\frac{c_p^r(H)}{N^p}=\frac{c_r^p(H^t)}{N^r}$$

But this gives the formula in the statement, and we are done.
\end{proof}

The above result shows that the normalized moments $\gamma_p^r=\frac{c_p^r}{N^p}$ are subject to the condition $\gamma_p^r(H)=\gamma_r^p(H^t)$. We have the following table of $\gamma_p^r$ numbers for $H$:
$$\begin{bmatrix}
p\backslash r&1&2&r&\infty\\
1&1/N&1/N&1/N&1/N\\
2&1/N&tr(S/N)^2&tr(S/N)^r&c_2\\
p&1/N&tr(S/N)^p&?&c_p\\
\infty&1/N&c_2&\mu^r(1)&\mu(1)
\end{bmatrix}$$

Here we have used the well-known fact that for $supp(\mu)\subset[0,1]$ we have $c_p\to\mu(1)$, something which is clear for discrete measures, and for continuous measures too.

Since the table for $H^t$ is transpose to the table of $H$, we obtain:

\begin{proposition}
$\mu_H(1)=\mu_{H^t}(1)$.
\end{proposition}

\begin{proof}
This follows indeed from Theorem 4.4, by letting $p,r\to\infty$.
\end{proof}

Observe that this result recovers a bit of Theorem 2.3, because we have:

\begin{proposition}
For $G\subset S_N^+$ finite we have $\mu(1)=\frac{1}{|G|}$.
\end{proposition}

\begin{proof}
The idea is to use the principal graph. So, let first $\Gamma$ be an arbitrary finite graph, with a distinguished vertex denoted $1$, let $A\in M_M(0,1)$ with $M=|\Gamma|$ be its adjacency matrix, set $N=||\Gamma||$, and let $\xi\in\mathbb R^M$ be a Perron-Frobenius eigenvector for $A$, known to be unique up to multiplication by a scalar. Our claim is that we have:
$$\lim_{p\to\infty}\frac{(A^p)_{11}}{N^p}=\frac{\xi_1^2}{||\xi||^2}$$

Indeed, if we choose an orthonormal basis of eigenvectors $(\xi^i)$, with $\xi^1=\xi/||\xi||$, and write $A=UDU^t$, with $U=[\xi^1\ldots\xi^M]$ and $D$ diagonal, then we have, as claimed:
\begin{eqnarray*}
(A^p)_{11}
&=&(UD^pU^t)_{11}=\sum_kU_{1k}^2D^p_{kk}\simeq U_{11}^2N^p=\frac{\xi_1^2}{||\xi||^2}N^p
\end{eqnarray*}

Now back to our quantum group $G\subset S_N^+$, let $\Gamma$ be its principal graph, having as vertices the elements $r\in Irr(G)$. The moments of $\mu$ being the numbers $c_p=(A^p)_{11}$, we have:
$$\mu(1)=\lim_{p\to\infty}\frac{c_p}{N^p}=\lim_{p\to\infty}\frac{(A^p)_{11}}{N^p}=\frac{\xi_1^2}{||\xi||^2}$$

On the other hand, it is known that with the normalization $\xi_1=1$, the entries of the Perron-Frobenius eigenvector are simply $\xi_r=\dim(r)$. Thus we have:
$$\frac{\xi_1^2}{||\xi||^2}=\frac{1}{\sum_r\dim(r)^2}=\frac{1}{|G|}$$

Together with the above formula of $\mu(1)$, this finishes the proof.
\end{proof}

\section{Deformed Fourier matrices}

In this section we study the deformed Fourier matrices, $L=F_M\otimes_QF_N$, constructed by Di\c t\u a in \cite{dit}. These matrices are defined by $L_{ia,jb}=Q_{ib}(F_M)_{ij}(F_N)_{ab}$. 

We first have the following technical result:

\begin{proposition}
Let $H=F_M\otimes_QF_N$, and set $R_{ab,cd}^x=\frac{1}{M}\sum_mw^{mx}\frac{Q_{ma}Q_{md}}{Q_{mc}Q_{mb}}$.
\begin{enumerate}
\item $Q_{iajb,kcld}=\delta_{a-b,c-d}R_{ab,cd}^{i+l-k-j}$.

\item $X_{i_1a_1\ldots i_ra_r,j_1b_1\ldots j_rb_r}=\delta_{a_1-b_1,\ldots,a_r-b_r}R_{a_1b_1,a_2b_2}^{i_1+j_2-j_1-i_2}\ldots R_{a_rb_r,a_1b_1}^{i_r+j_1-j_r-i_1}$.
\end{enumerate}
\end{proposition}

\begin{proof}
First, for a general deformation $H=K\otimes_QL$, we have:
\begin{eqnarray*}
Q_{iajb,kcld}
&=&\frac{1}{MN}\sum_{me}\frac{H_{me,ia}H_{me,ld}}{H_{me,kc}H_{me,jb}}
=\frac{1}{MN}\sum_{me}\frac{Q_{ma}K_{mi}L_{ea}Q_{md}K_{ml}L_{ld}}{Q_{mc}K_{mk}L_{ec}Q_{mb}K_{mj}L_{eb}}\\
&=&\frac{1}{M}\sum_m\frac{Q_{ma}Q_{md}}{Q_{mc}Q_{mb}}\cdot\frac{K_{mi}K_{ml}}{K_{mk}K_{mj}}\cdot\frac{1}{N}\sum_e\frac{L_{ea}L_{ed}}{L_{ec}L_{eb}}
\end{eqnarray*}

Thus for a deformed Fourier matrix $H=F_M\otimes_QF_N$ we have:
$$Q_{iajb,kcld}=\delta_{a+d,b+c}\frac{1}{M}\sum_m\frac{Q_{ma}Q_{md}}{Q_{mc}Q_{mb}}w^{m(i+l-k-j)}$$

But this gives (1), and then (2), and we are done.
\end{proof}

With the above formulae in hand, we can now state and prove:

\begin{theorem}
For the matrix $H=F_M\otimes_QF_N$ we have
$$\mu_{H}=\mu_{H^t}$$
for any value of the parameter matrix $Q\in M_{M\times N}(\mathbb T)$.
\end{theorem}

\begin{proof}
We use the matrices $X,R$ constructed in Proposition 5.1 above. According to the result in Proposition 5.1 (2), we have the following formula:
\begin{eqnarray*}
c_p^r
&=&\frac{1}{N^r}\sum_{a_1^1\ldots a_p^r}X_{a_1^1\ldots a_1^r,a_2^1\ldots a_2^r}\ldots\ldots X_{a_p^1\ldots a_p^r,a_1^1\ldots a_1^r}\\
&=&\frac{1}{N^r}\sum_{a_1^1\ldots a_p^r}\sum_{i_1^1\ldots i_p^r}\delta_{a_1^1-a_2^1,\ldots,a_1^r-a_2^r}R_{a_1^1a_2^1,a_1^2a_2^2}^{i_1^1+i_2^2-i_1^2-i_2^1}\ldots\ldots R_{a_1^ra_2^r,a_1^1a_2^1}^{i_1^r+i_2^1-i_1^1-i_2^r}\\
&&\ldots\\
&&\delta_{a_p^1-a_1^1,\ldots,a_p^r-a_1^r}R_{a_p^1a_1^1,a_p^2a_1^2}^{i_p^1+i_1^2-i_1^1-i_p^2}\ldots\ldots R_{a_p^ra_1^r,a_p^1a_1^1}^{i_p^r+i_1^1-i_p^1-i_1^r}
\end{eqnarray*}

Observe that the conditions on the $a$ indices, coming from the Kronecker symbols, state that the columns of $a=(a_i^j)$ must differ by vertical vectors of type $(s,\ldots,s)$.

Now let us compute the sum over $i$ indices, obtained by neglecting the Kronecker symbols. According to the formula of $R_{ab,cd}^x$ in Proposition 5.1, this is:
\begin{eqnarray*}
S
&=&\frac{1}{N^{pr}}\sum_{i_1^1\ldots i_p^r}\sum_{m_1^1\ldots m_p^r}w^{E(i,m)}\frac{Q_{m_1^1a_1^1}Q_{m_1^1a_2^2}}{Q_{m_1^1a_2^1}Q_{m_1^1a_1^2}}\ldots\ldots\frac{Q_{m_1^ra_1^r}Q_{m_1^ra_2^1}}{Q_{m_1^ra_2^r}Q_{m_1^ra_1^1}}\\
&&\ldots\\
&&\frac{Q_{m_p^1a_p^1}Q_{m_p^1a_1^2}}{Q_{m_p^1a_1^1}Q_{m_p^1a_p^2}}\ldots\ldots\frac{Q_{m_p^ra_p^r}Q_{m_p^ra_1^1}}{Q_{m_p^ra_1^r}Q_{m_p^ra_p^1}}
\end{eqnarray*}

Here the exponent appearing at right is given by:
\begin{eqnarray*}
E(i,m)
&=&m_1^1(i_1^1+i_2^2-i_1^2-i_2^1)+\ldots+m_1^r(i_1^r+i_2^1-i_1^1-i_2^r)\\
&&\ldots\\
&&+m_p^1(i_p^1+i_1^2-i_p^2-i_1^1)+\ldots+m_p^r(i_p^r+i_1^1-i_p^1-i_1^r)
\end{eqnarray*}

Now observe that this exponent can be written as:
\begin{eqnarray*}
E(i,m)
&=&i_1^1(m_1^1-m_1^r-m_p^1+m_p^r)+\ldots+i_1^r(m_1^r-m_1^{r-1}-m_p^r+m_p^{r-1})\\
&&\ldots\\
&&+i_p^1(m_p^1-m_p^r-m_{p-1}^1+m_{p-1}^r)+\ldots+i_p^r(m_p^r-m_p^{r-1}-m_{p-1}^r+m_{p-1}^{r-1})
\end{eqnarray*}

With this formula in hand, we can perform the sum over the $i$ indices, and the point if that the resulting condition on the $m$ indices will be exactly the same as the above-mentioned condition on the $a$ indices. Thus, we obtain a formula as follows, where $\Delta(.)$ is a certain product of Kronecker symbols:
\begin{eqnarray*}
c_p^r
&=&\frac{1}{N^r}\sum_{a_1^1\ldots a_p^r}\sum_{m_1^1\ldots m_p^r}\Delta(a)\Delta(m)\frac{Q_{m_1^1a_1^1}Q_{m_1^1a_2^2}}{Q_{m_1^1a_2^1}Q_{m_1^1a_1^2}}\ldots\ldots\frac{Q_{m_1^ra_1^r}Q_{m_1^ra_2^1}}{Q_{m_1^ra_2^r}Q_{m_1^ra_1^1}}\\
&&\ldots\\
&&\frac{Q_{m_p^1a_p^1}Q_{m_p^1a_1^2}}{Q_{m_p^1a_1^1}Q_{m_p^1a_p^2}}\ldots\ldots\frac{Q_{m_p^ra_p^r}Q_{m_p^ra_1^1}}{Q_{m_p^ra_1^r}Q_{m_p^ra_p^1}}
\end{eqnarray*}

The point now is that when replacing $H=F_M\otimes_QF_N$ with its transpose matrix, $H^t=F_N\otimes_{Q^t}F_M$, we will obtain exactly the same formula, with $Q$ replaced by $Q^t$. But, with $a_x^y\leftrightarrow m_x^y$, this latter formula will be exactly the one above, and we are done.
\end{proof}

\end{document}